\documentclass[12pt]{article}
\usepackage{amssymb, amsmath}

\newcommand{\qed}{\hskip 5mm \rule{2.5mm}{2.5mm}}
\newcommand{\R}{{\mathbb R}}

\newcommand{\N}{{\mathbb N}}

\newcommand{\proof}{\noindent{\em Proof:\ }}

\begin{document}
\newtheorem{thm}{Theorem}[section]
\newtheorem{defs}[thm]{Definition}
\newtheorem{lem}[thm]{Lemma}
\newtheorem{note}[thm]{Note}
\newtheorem{cor}[thm]{Corollary}
\newtheorem{prop}[thm]{Proposition}
\renewcommand{\theequation}{\arabic{section}.\arabic{equation}}
\newcommand{\newsection}[1]{\setcounter{equation}{0} \section{#1}}
%%%%%%%%%%%%%%%%%%%%%%%%%% title %%
\title{A Hahn-Jordan decomposition and Riesz-Frechet representation theorem in Riesz spaces
      }
%%%%%%%%%%%%%%%%%%%%%%%%%%%%%%%%%%%%%%%%%%%%%%
\author{Anke Kalauch\\
Institut für Analysis, FR Mathematik,\\
TU Dresden,\\
D-01062 Dresden, Germany\\ \\
Wenchi Kuo\footnote{Supported in part NRF grant RA191123493824 number 132053.} \& 
Bruce A. Watson \footnote{Supported in part NRF grant SATN180717350298 number 120112 and the Centre for Applicable Analysis and Number Theory.} \\ 
School of Mathematics\\
University of the Witwatersrand\\
Private Bag 3, P O WITS 2050, South Africa}

\maketitle
%%%%%%%%%%%%%%%%%%%%%%%%%%%%%%%%%%%%%%%%
\abstract{
 We give a Hahn-Jordan decomposition in Riesz spaces which generalizes that of [{{\sc B. A. Watson}, {An And\^o-Douglas type theorem in Riesz spaces with a conditional expectation,} {\em Positivity,} {\bf 13} (2009), 543 - 558}] and a Riesz-Frechet representation theorem for the $T$-strong dual, where $T$ is a Riesz space conditional expectation operator.  
The result of Watson was formulated specifically  to assist in the proof of the existence of Riesz space conditional expectation operators with given range space, i.e., a result of And\^{o}-Douglas type. This was needed in the study of Markov processes and martingale theory in Riesz spaces. In the current work, our interest is a Riesz-Frechet representation theorem, for which another variant of the Hahn-Jordan decomposition is required.
  }
%%%%%%%%%%%%%%%%%%%%%%%%%%
\parindent=0in
\parskip=.2in
\newsection{Introduction}
We present a generalization of the Hahn-Jordan decomposition of measure theory to an abstract form in Riesz spaces with two objectives: first, that it include as an application the Hahn-Jordan decomposition givcn in \cite{watson}; second, that it provide the generalization needed in duality theory for the proof of a Riesz-Frechet representation theorem in Riesz spaces, see later for details. In \cite{watson}, the aim was to prove the existence of Riesz space conditional expectation operators with suitably specified  range spaces. In particular a generalization of the results of And\^{o} and Douglas \cite{ando, douglas} was proved. This provided a critical result for the study of Markov processes in Riesz spaces and martingale theory in Riesz spaces, see \cite{grobler, vardy-watson}. For other results in this area we refer the reader to \cite{h-dp, kusraev}.

The setting used for our Hahn-Jordan decomposition is as follows.
Let $E$ be a Dedekind complete Riesz space with weak order unit, say $e$, and $G$ be a Dedekind complete Riesz subspace of $E$ which also contains $e$,
so $\{e\}\subset G \subset E$.
Denote  the set of all components of $e$ in $E$ by  ${\cal K}(e)$, and by ${\cal K}_G(e)=G\cap {\cal K}(e)$  the set of all components of $e$ in $E$ which are in $G$. 
It should be noted here that ${\cal K}(e)$ and ${\cal K}_G(e)$ are Boolean algebras with ${\cal K}_G(e)$ a subalgebra of ${\cal K}(e)$. 
The multiplication of the components of $e$ is that of the $f$-algebra $E_e$ which is taken to have algebraic unit $e$. Here $E_e$ denotes the ideal in $E$ generated by $e$, i.e., $E_e=\{f\in E\,|\, |f|\le ke \mbox{ for some } k\in\R^+\}$. 
Moreover $p,q$ are components of $e$ in $E$ if and only if there are band projections 
$P,Q$ on $E$ with $p=Pe$ and $q=Qe$. In this case, $pq=PQe=QPe=qp$ and $pq=p\wedge q$. It should be noted that $E$ is an $E_e$-module, see \cite{KLW-exp,KRW}. 
As $e$ is a weak order unit for $G$, we have that $G$ and $E$ are $G_e$-modules. 
Here ${\cal K}_G(e)\subset G_e$.

Let $\cal B$ be a Boolean subalgebra of ${\cal K}(e)$ which contains ${\cal K}_G(e)$ and is order closed in $E$, i.e.,  if $(p_\alpha)$ is a net in ${\cal B}$ with $p_\alpha\downarrow p$ in $E$, then $p\in {\cal B}$.
Take $\psi:{\cal B}\to G$ be a map such that:
\begin{enumerate}
\item[(i)] If $p\in {\cal B}$ and $q\in {\cal K}_G(e)$ then $\psi(pq)=q\psi(p)$.
\item[(ii)] If $p,q\in {\cal B}$ with $pq=0$ then $\psi(p\vee q)=\psi(p+q)=\psi(p)+\psi(q)$ (additivity).
\item[(iii)] If $(p_\alpha)$ is a net in ${\cal B}$ with $p_\alpha\downarrow p$ in $E$, then $\psi(p_\alpha)\to \psi(p)$ (order continuity of $\psi$).
\item[(iv)] There is $g\in E^+$ so that $|\psi(p)|\le g$ for all $p\in{\cal B}$ ($\psi$ is order bounded).
\end{enumerate}

We say that $q\in {\cal B}$ is strongly positive (resp. strongly negative) with respect to $\psi$ if $\psi(p)\ge 0$ (resp. $\le 0$) for all $p\in {\cal B}$ with $p\le q$. 
The Hahn-Jordan decomposition presented in Theorem \ref{Hahn} gives the existence of $q\in {\cal B}$ so that $q$ is strongly positive with respect to $\psi$ and $e-q$ is strongly negative with respect to $\psi$.

As noted earlier, we apply the Hahn-Jordan decomposition, Theorem \ref{Hahn}, in the following setting to obtain a Riesz-Frechet representation of the $T$-strong dual of a Riesz space.
Let $E$ be a Dedekind complete Riesz space with weak order unit.
 By $T$ being a conditional expectation
operator on $E$ we mean that $T$ is a linear positive order continuous projection on $E$ which maps weak order units to weak order units and has range $R(T)$ closed with respect to order limits in $E$. This gives that there is at least one weak order unit, say $e$, with $Te=e$, and that 
$R(T)$ is Dedekind complete when considered as a subspace of $E$. 
By $T$ being strictly positive we mean that if
$f\in E^+$, the positive cone of $E$, and $f\ne 0$ then $Tf\in E^+$ and $Tf\ne 0$.

As shown in \cite{KLW-exp}, a strictly positive conditional expectation operator, $T$, on a Dedekind complete Riesz space with weak order unit, can be extended to a strictly positive conditional expectation operator, also denoted $T$,  on its natural domain, denoted $L^1(T):=\mbox{dom}(T)-\mbox{dom}(T)$.
We say that $E$ is $T$-universally complete if $E=L^1(T)$. From the definition of $\mbox{dom}(T)$, see \cite{KLW-exp}, $E$ is $T$-universally complete if and only
 if for each upwards directed net $(f_{\alpha})_{\alpha \in \Lambda}$ in $E^+$ such that $(Tf_{\alpha})_{\alpha \in \Lambda}$ is order bounded in $E_{u}$, we have that $(f_{\alpha})_{\alpha \in \Lambda}$ is order convergent in $E$. Here $E_u$ denotes the universal completion of $E$, see \cite[page 323]{L-Z}.  $E_u$ has an $f$-algebra structure which can be chosen so that $e$ is the multiplicative identity.

 Let $E$ be a $T$-universally complete Riesz space, where $T$ is a strictly positive conditional expectation   operator on 
  $E$, and let $e$ be a weak order unit for $E$ with $Te=e$.
  From \cite{KRW}  $R(T)$ is universally complete and an $f$-algebra, further  $E=L^1(T)$ is an $R(T)$-module.
 From \cite[Theorem 5.3]{KLW-exp}, $T$ is an averaging operator, i.e., if $f\in R(T)$ 
and $g\in E$ then $T(fg)=fT(g)$.
 This prompts the definition of an $R(T)$ (vector valued) norm $\|\cdot\|_{T,1}:=T|\cdot|$ on $L^1(T)$. 
The homogeneity is with respect to multiplication by elements of $R(T)^+$.
The Hahn-Jordan decomposition of \cite{watson} is an application of Theorem \ref{Hahn} with $\psi(p)=T(pf)$ for $p\in{\cal B}$ and $f\in E$, see Section \ref{HJD}.

The definition of the Riesz space
 $L^{2}(T):=\{f\in L^1(T)\,|\, f^2\in L^1(T)\}$ was given in  
 \cite{LW}.
 By the averaging property,  $L^{2}(T)$ is an $R(T)$-module and the map
 $$f\mapsto\|f\|_{T,2}:=(T(f^2))^{1/2}, \quad f\in L^{2}(T),$$
 is an $R(T)$-valued norm on $L^{2}(T)$. Aspects for this development for $L^{p}(T)$ with general  $1<p<\infty$ can be found in \cite{AT, grobler-1}. Here the multiplication is as defined in the $f$-algebra $E_u$, which when restricted to the $f$-algebra $E_e$ is also the multiplication there.
Proofs of various H\"older type inequalties in Riesz spaces with conditional expectation operators can be found in \cite{KRW} and \cite{AT}. In particular,
\begin{equation}\label{holder}
T|fg|\le \|f\|_{T,2}\|g\|_{T,2}, \, \mbox{ for all } f,g\in {{L}}^2(T).
\end{equation}

Let $E=L^{2}(T)$. We say that a map 
${\frak f}:E\to R(T)$
is a $T$-linear functional on $E$ if it is additive, $R(T)$-homogeneous and order continuous. We recall, from \cite[Theorem 4.10]{AB}, that, since $R(T)$ is a Dedekind complete Riesz space and $E$ is a Riesz space, a linear map from $E$ to $R(T)$ is order bounded if and only if it is order continuous. We denote the space of $T$-linear functionals on $E$ by $E^*$ and call it the $T$-dual of $E$. We note that $E^*\subset {\cal L}_b(E, R(T))$, since $R(T)$-homogeneity implies real linearity. Further as $R(T)$ is Dedekind complete, so is ${\cal L}_b(E, R(T))$, see \cite[page 12]{AB}. 

If ${\frak f}\in E^*$ and 
there is $k\in R(T)^+$ such that 
$$|{\frak f}(g)|\le k\|g\|_{T,2},\quad\mbox{for all } g\in E,$$
we say that ${\frak f}$ is $T$-strongly bounded. 
We denote the space of $T$-strongly bounded $T$-linear functionals on $E$ by 
\begin{eqnarray*}
\hat{E}:=\{{\frak f}\in E^*\,|\, {\frak f} \mbox{ $T$-strongly bounded}\}
\end{eqnarray*}
and refer to it as the $T$-strong dual of $E$. 
Further,
$$\|{\frak f}\|:=\inf\{k\in R(T)^+\,|\,|{\frak f}(g)|\le k\|g\|_{T,2}\quad\mbox{for all } g\in E\}$$
defines an $R(T)$-valued norm on $\hat{E}$ with 
\begin{eqnarray}\label{norm-bound}
|{\frak f}(g)|\le \|{\frak f}\|\,\|g\|_{T,2}
\end{eqnarray}
 for all $g\in L^2(T)$.

 Applying Theorem \ref{Hahn} in this setting (with $\psi(p)={\frak f}(p)$), see Lemma \ref{hj-dual-lem} below, leads to a Riesz-Frechet representation of $\hat{E}$, see Theorem \ref{R-F}. This, when combined with Proposition \ref{dual-map}, gives the following identification. 
 
\begin{thm}\label{thm-final}
 The map $\Psi$ defined by $\Psi(f)(g):=T_f(g)=T(fg)$ for $f,g\in {{L}}^2(T)$ is a bijection between $E={{L}}^2(T)$ and, its $R(T)$-homogeneous strong dual,
 $\hat{E}$. This map is additive, $R(T)$-homogeneous 
 and $R(T)$-valued norm preserving in the sense that $\|T_f\|=\|f\|_{T,2}$ for all $f\in {{L}}^2(T)$.
\end{thm}

 In the Appendix we present the required concept of a partial inverse in Riesz spaces along with some of its properties.

%%%%%%%%%%%%%%%%%%%%%%%%%%%%%%%%%%%%%%%
\newsection{Preliminaries}\label{prelim}
For essential background in Riesz space theory we refer the reader to \cite{Abr-Ali, AB, L-Z, zaanen-2, zaanen}. First we discuss properties of the Boolean subalgebra ${\cal B}$ of ${\cal K}(e)$ and the map
$\psi:{\cal B}\to G$ with the conditions (i) to (iv) given in the introduction.

Concerning condition  (i), $q\in {\cal K}_G(e)$ gives that $q\in{\cal B}$, and ${\cal B}$ is a Boolean algebra, so $pq\in {\cal B}$.  Thus $\psi(pq)$ is defined and is in $G$. Further, $\psi(p)\in G$ and 
$G$ is a $G_e$ module, but $q\in {\cal K}_G(e)\subset G_e$ so ensuring that $q\psi(p)$ is defined in $G$.
By (i) with $q=0$ we have $\psi(0)=0$.
With regards to condition (iii), we note that it also holds for upward directed nets.

For $q\in {\cal B}$ we set 
$$C(q):=\{\psi(p)\,|\,p\in{\cal B}, p\le q\}.$$
Since $0\in {\cal B}$, $0\le q$ and $\psi(0)=0$ we have that $0\in C(q)$. Further, as $\psi$ is order bounded, so is $C(q)$.
Thus, as $G$ is Dedekind complete, we can define 
$$\alpha(q):=\sup C(q),$$
and here $\alpha(q)\in G^+$, since $0\in C(q)$.

 For $q\in{\cal B}$ let 
 $${\cal M}(q):= \{p\in{\cal B}\ |\ 2\psi(p)\ge p\alpha(q), p\le q\}.$$
It should be note that the $2$ in the above inequality (and the resultant work below) could be replaced by any fixed real number greater than $1$. For all $p\le q$ we have $p\psi(p)\le p\alpha(q)$, but ${\cal M}(q)$ gives those $p\le q$ for which $p\psi(p)$ is close to $p\alpha(q)$ (to within a given multiplicative factor of $p\alpha(q)$, which we have chosen as $1/2$).

\begin{lem}\label{maximal}
 For each $q\in{\cal B}$,  $0\in{\cal M}(q)$ and ${\cal M}(q)$ has a maximal element $\hat{q}\in{\cal B}$.
 \end{lem}
 
\proof
As $0\in G\subset {\cal B}$, $0\le q$ and $2\varphi(0)=0=0\alpha(q)$ it follows that $0\in {\cal M}(q)$.

 Let $(p_\gamma)$ be a totally ordered  increasing net in ${\cal M}(q)$ and $p_\gamma\le e$, so, from the Dedekind completeness of $E$,
 $p_\gamma\to \bar{p}:=\sup_\gamma p_\gamma\in E$. Since ${\cal B}$ is order closed in $E$, $\bar{p}\in{\cal B}$.
  Further, from the order continuity of $\psi$, we have that 
  $$\bar{p}\alpha(q)\leftarrow p_\gamma\alpha(q)\le 2\psi(p_\gamma)\to 2\psi(\bar{p})$$
  in order. Thus
  $\bar{p}\alpha(q)\le 2\psi(\bar{p})$. Also $\bar{p}\leftarrow p_\gamma\le q$ in order, so $\bar{p}\le q$. Hence $\bar{p}\in{\cal M}(q)$.
   Thus $\bar{p}$ is an upper bound for $(p_\gamma)$ from ${\cal M}(q)$ and Zorn's lemma gives that ${\cal M}(q)$
 has a maximal element, say $\hat{q}$.
\qed

\begin{lem}\label{disjointness}
 For $q\in{\cal B}$ and $\hat{q}$ a maximal element of ${\cal M}(q)$, 
 $$(\alpha(q)-2\psi(\hat{q}))^+\hat{q}=0,$$
 i.e.,  $(\alpha(q)-2\psi(\hat{q}))^+$ and $\hat{q}=0$ are disjoint in $E$ and $\hat{q}(\alpha(q)-2\psi(\hat{q}))\le 0$.
 \end{lem}

\proof
  Let $Q$ denote the band projection onto the band, in $E$,  generated by 
 $\left(\alpha(q)-2\psi(\hat{q})\right)^+$ and let $k:=Qe$, 
 then $k$ is a component of $e$.
Now
\begin{equation}\label{2021-04-16-ineq-1}
k(2\psi(\hat{q})-\alpha(q))=-k(\alpha(q)-2\psi(\hat{q}))=-(\alpha(q)-2\psi(\hat{q}))^+\le 0.
\end{equation} 
 As $\hat{q}\in {\cal M}(q)$, it follows that $\hat{q}\hat{q}=\hat{q}\ge 0$ and
 $2\psi(\hat{q})\ge \hat{q}\alpha(q)$.
 Thus
\begin{equation}\label{2021-04-16-ineq-2}
\hat{q}(2\psi(\hat{q})-\alpha({q}))=\hat{q}(2\psi(\hat{q})-\hat{q}\alpha({q}))\ge 0.
\end{equation}  
Combining (\ref{2021-04-16-ineq-1}) and (\ref{2021-04-16-ineq-2}) gives
\begin{equation}\label{2021-04-16-ineq-3}
0\ge \hat{q}k(2\psi(\hat{q})-\alpha({q}))=k\hat{q}(2\psi(\hat{q})-\hat{q}\alpha({q}))\ge 0.
\end{equation}  
Thus $$0=\hat{q}k(2\psi(\hat{q})-\alpha({q}))=\hat{q}(2\psi(\hat{q})-\alpha({q}))^+$$
and the disjointness follows.
\qed

\begin{lem}\label{existence-1}
  Let $Q$ denote the band projection onto the band, in $E$,  generated by 
 $\left(\alpha(q)-2\psi(\hat{q})\right)^+$ and let $k:=Qe$, 
 then $k\in {\cal K}_G(e)\subset {\cal B}$. 
 Further,  if $k\alpha(q)>0$, then
 there exists $p_0\in {\cal B}$ with $p_0\le q$ such that
 $$k(2\psi(p_0)-\alpha(q))\not\le 0.$$
\end{lem}

\proof
Since  $\alpha(q), \psi(\hat{q})\in G$, we have  $k\in G$ and thus $k\in {\cal B}$. 

 Suppose $k\alpha(q)>0$.  If there does not exist $p_0\in {\cal B}$ with $p_0\le q$ such that
 $$k(2\psi(p_0)-\alpha(q))\not\le 0,$$
 then, for all  $p\in {\cal B}$ with $p\le q$, we have
 $2k\psi(p)\le k\alpha(q).$
Taking the supremum over  $p\in {\cal B}$ with $p\le q$ in the above gives
 $2k\alpha(q)\le k\alpha(q)$. Thus $k\alpha(q)\le 0$ which contradicts the assumption $k\alpha(q)>0$.
 \qed

\begin{thm}\label{lower}
 For each $q\in{\cal B}$, if $\hat{q}$ is a maximal element of ${\cal M}(q)$, then $\hat{q}\in {\cal B}$ with $\hat{q}\le q$ and 
 $\alpha(q)\le 2\psi(\hat{q})\le 2\alpha(q)$.
\end{thm}

\proof
From the definition of $\alpha(q)$, as $\hat{q}\in{\cal M}(q)$, it follows that $2\psi(\hat{q})\le 2\alpha(q)$.

 We now show that $2\psi(\hat{q})\ge \alpha({q})$. 
If this were not the case, then
 $\left(\alpha(q)-2\psi(\hat{q})\right)^+>0$.
  Let $Q$ denote the band projection onto the band, in $E$,  generated by 
 $\left(\alpha(q)-2\psi(\hat{q})\right)^+$ and $k:=Qe>0$.
 As $k\in {\cal K}_G(e)$ so $\psi(k\hat{q})=k\psi(\hat{q})$.  From Lemma \ref{disjointness},  $0=k\hat{q}$,  thus, as 
 $\psi(0)=0$, we have
 $$k\psi(\hat{q})=\psi(k\hat{q})=\psi(0)=0.$$
 Hence
 $$k\alpha(q)=k\left(\alpha(q)-2\psi(\hat{q})\right)=\left(\alpha(q)-2\psi(\hat{q})\right)^+> 0.$$
 By Lemma \ref{existence-1},
 there exists $p_0\in {\cal B}$ with $p_0\le q$ such that
 \begin{equation*}
 k(2\psi(p_0)-\alpha(q))\not\le 0,
 \end{equation*}
 and thus
 \begin{equation}\label{notle}
 k(2\psi(p_0)-\alpha(q))^+=  (k(2\psi(p_0)-\alpha(q)))^+> 0,
 \end{equation}
 as $k\ge 0$.
 
Let $P_1$ be the band projection onto the band generated by $(2\psi(p_0)-\alpha(q))^+$ and $p_1=P_1e\in {\cal K}_G(e)$.
  Let $\bar{q}:=kp_1\in {\cal K}_G(e)$,  then from (\ref{notle}),  $\bar{q}>0$ and
  \begin{equation}\label{pos-1}
   \bar{q}(2\psi(p_0)-\alpha(q))=kp_1(2\psi(p_0)-\alpha(q))=k(2\psi(p_0)-\alpha(q))^+> 0.
\end{equation}
    Further, as $\alpha(q)\ge 0$ and $0\le p_0\le e$,  we have
  $p_0\alpha(q)\le \alpha(q)$. Thus, from equation (\ref{pos-1}) and the definition of $\bar{q}$,
  \begin{equation}\label{new-plus}
   2\bar{q}\psi(p_0)>\bar{q}\alpha(q)\ge \bar{q}p_0\alpha(q).
 \end{equation}
  Here $k, p_1, \bar{q}\in{\cal K}_G(e)$ and $p_0\in {\cal B}$ so  $\bar{q}\psi(p_0)=\psi(\bar{q}p_0)$ and
  $\bar{q}p_0\in {\cal B}$,  thus 
  from (\ref{new-plus}),
  \begin{equation}\label{14-may}
  2\psi(\bar{q}p_0)> \bar{q}p_0\alpha(q),
  \end{equation}
giving   $\bar{q}p_0\in{\cal M}(q)$. Here $\bar{q}p_0>0$, as otherwise (\ref{14-may}) would give $0>0$.
Further
$$\bar{q}p_0=kp_1p_0\le kp_0\le kq$$
and, by Lemma \ref{disjointness},  $k\hat{q}=0$ giving $k\le e-\hat{q}$.  Thus
$$\bar{q}p_0\le q(e-\hat{q})=q-q\hat{q}=q-\hat{q},$$
where we have used that $q\hat{q}=\hat{q}$, since $\hat{q}\le q$. 
Multiplying by $\hat{q}$ gives
$$\hat{q}\bar{q}p_0\le \hat{q}(q-\hat{q})=0,$$
so $\hat{q}\bar{q}p_0=0$ making $\hat{q}$ and $\bar{q}p_0$ disjoint elements of ${\cal B}$.  Hence $\hat{p}=\hat{q}+\bar{q}p_0=\hat{q}\vee\bar{q}p_0\in {\cal B}$ and $q\ge\hat{p}>\hat{q}$.
From Lemma \ref{disjointness}, $\hat{q}(\alpha(q)-2\psi(\hat{q}))\le 0$,  which, together with (\ref{14-may}), gives
$$2\psi(\hat{p})=2\psi(\hat{q})+2\psi(\bar{q}p_0)\ge \hat{q}\alpha(q)+\bar{q}p_0\alpha(q)=\hat{p}\alpha(q),$$
so $\hat{p}\in{\cal M}(q)$, which contradicts  maximality of $\hat{q}$. Thus
 $2\psi(\hat{q})\ge \alpha({q})$. 
  \qed
      
\begin{cor}\label{cor-neg}
 For each $q\in{\cal B}$ there exists $u\in {\cal B}$ with $u\le  P_{\psi(u)}q$, $\psi(u)\ge 0$ and 
 $\psi(u)\le\alpha(q)\le 2\psi(u)$. Here $P_{\psi(u)}$ is the band projection onto the band in $E$ generated by $\psi(u)$.
\end{cor}

\proof
Let $q\in{\cal B}$ and $\hat{q}$ be a maximal element of ${\cal M}(q)$.
 From Theorem \ref{lower},  $\hat{q}\in {\cal B}$, $\hat{q}\le q$ and $\psi(\hat{q})\le\alpha(q)\le 2\psi(\hat{q})$.
Since $\alpha(q)\ge 0$, it follows that $\psi(\hat{q})\ge 0$ and the bands generated by $\psi(\hat{q})$ and $\alpha(q)$ coincide, giving $P_{\psi(\hat{q})}\alpha(q)=\alpha(q)$.
Let $g:=P_{\psi(\hat{q})}e$, then $g\alpha(q)=\alpha(q)$, $g\in {\cal K}_G(e)$ and $\hat{q}\in{\cal B}$ so $\psi(g\hat{q})=g\psi(\hat{q})$.
Setting $u:=g\hat{q}$, we obtain
\begin{equation}\label{psi-inequality}
2\psi(u)
 =2\psi(g \hat{q})
 =2g \psi(\hat{q})
 \ge g \alpha(q)=\alpha(q).
 \end{equation}
As $u\le \hat{q}\le q$ it follows that $\psi(u)\le \alpha(q)$, which combined with (\ref{psi-inequality}) gives $\psi(u)\le\alpha(q)\le 2\psi(u)$ and $\psi(u)\ge 0$. 
Finally, as the bands generated by $\psi(u)$ and $\alpha(q)$ coincide and this coincides with the band generated by $\psi(\hat{q})$, we have
$$u=g\hat{q}=P_{\psi(u)}\hat{q}\le P_{\psi(u)}q,$$
since $\hat{q}\le q$. 
\qed

%%%%%%%%%%%%%%%%%%%%%%%%%%%%%
\newsection{Hahn-Jordan decomposition}\label{HJD}

We are now in a position to develop the generalized abstract Hahn-Jordan decomposition, which generalizes \cite[Theorem  3.5]{watson} (details of this are given as an application at the end of his section) and \cite[page 184]{zaanen}.

\begin{thm}\label{strongly}
 If $p\in\mathcal{B}$ with $\psi(p)<0$ then there exists $v\in\mathcal{B}$ with 
 $v\le p$ and $v$ strongly negative with respect to $\psi$ and $\psi(v)\le \psi(p)$.
 \end{thm}

\proof
 Let 
 $$\alpha_1:=\alpha(p)\ge 0.$$
 From Corollary \ref{cor-neg}, there is  $p_1\in {\cal B}$ with 
 $p_1\le P_{\psi(p_1)^+}p$ and  $\alpha_1\le 2\psi(p_1)$.
 We define $\alpha_{n+1}$ and $p_{n+1}$ inductively for $n\in\N$ by
 $$\alpha_{n+1}:=\alpha\left(p-\bigvee_{i=1}^n p_i\right)$$
 and, using Corollary \ref{cor-neg}, we take
 $p_{n+1}\in {\cal B}$ with
 $$p_{n+1} \le P_{\psi(p_{n+1})^+}\left(p-\bigvee_{i=1}^n p_i\right)$$ and 
 $\alpha_{n+1}\le 2\psi(p_{n+1}).$
 Here 
 $\displaystyle{p_{n+1}\le p-\bigvee_{i=1}^n p_i}$,
 so $p_{n+1}p_i=0$, for all $i=1,\dots,n$, and thus $p_ip_j=0$ for all $i\ne j$. 

 Since ${\cal B}$ is order closed and $E$ is Dedekind complete, we can define
 $$\bar{p}:=\sum_{i=1}^\infty p_i=\bigvee_{i=1}^\infty p_i\in{\cal B}.$$
 Here $\bar{p}\le p$.
 The additivity and order continuity of $\psi$ give that
 $$\psi(\bar{p})= \psi\left(\sum_{i=1}^\infty p_i \right)=\sum_{i=1}^\infty \psi(p_i),$$
which, along with the definition of $p_n$, gives 
 \begin{eqnarray}
  0\le \sum_{i=1}^\infty \alpha_i \le \sum_{i=1}^\infty 2\psi({p}_i)  = 2\psi(\bar{p}).\label{new-2007}
 \end{eqnarray}
 Let $v:=p-\bar{p}$ then $v\le p$ and
 $\psi(v)+\psi(\bar{p})=\psi(p)$, but $\psi(\bar{p})\ge 0$ giving
 $\psi(v)\le \psi(p)$. 
 If  $q\le v$ then $q\le p-\sum_{i=1}^n p_i$ making
 $\psi(q)\le \alpha_n$, for all $n\in\N$. This together with (\ref{new-2007}) and $g\ge |\psi(\bar{p})|$
gives that 
 $$n\psi(q) \le \sum_{i=1}^n \alpha_i\le 2\psi(\bar{p}) \le 2g,$$
 for all $n\in\N$. Since $E$ is Archimedean, this gives that $\psi(q)\le 0$. Hence $v$ is strongly negative with respect to $\psi$ and $v$ is as required by the theorem.
\qed

\begin{cor}\label{existence}
  If  $\psi(q)\not\ge 0$ for some $q\in{\cal B}$,
 then there exists $v\in {\cal B}$ with $v\le q$ such 
 that $v$ is strongly negative with respect to $\psi$ and $\psi(v)\le -\psi(q)^-$. 
\end{cor}

\proof
Let $g:=P_{(\psi(q))^-}e$, then $g\in {\cal K}_G(e)$.
 Let $p:=gq\in{\cal B}$ then $\psi(p)=g\psi(q)=-\psi(q)^-<0$. 
 Applying the previous theorem we have that there is
 $v\in{\cal B}$ with $v\le p=gq\le q$ so that $v$ is strongly negative with respect to $\psi$ and
 $\psi(v)\le\psi(p) =-\psi(q)^-$.
 \qed

We are now in a position to conclude the proof of the Hahn-Jordan decomposition of $\psi$ with respect to ${\cal B}$.

\begin{thm}\label{Hahn}{\bf Abstract Hahn-Jordan Decomposition}\\
 There exists  $q\in{\cal B}$ which is strongly positive with respect to $\psi$
 and which has $e-q$ strongly negative with respect to $\psi$.
\end{thm}

\proof
 If $\psi(p)\ge 0$ for all $p\in{\cal B}$  take $q=e$ or (respectively) $q=0$, completing the proof for this case. 
 
 Assume that there exists $p\in{\cal B}$ with $\psi(p)\not\ge 0$.
Let
$${\cal H}:=\{ p\in{\cal B}\ |\ P\ \mbox{strongly negative w.r.t.}\ \psi \},$$
then $0\in{\cal H}$, so ${\cal H}$ is non-empty.
By the order continuity of $\psi$ and the order closedness of ${\cal B}$, ${\cal H}$ is order closed with respect to limits of upwards directed nets. Hence, by Zorn's lemma, ${\cal H}$ has maximal elements.
If $p_1, p_2 \in {\cal H}$ and $q\in {\cal B}$, then $p_1p_2q, (e-p_1)p_2q, (e-p_2)p_1q \in {\cal H}$ and thus $\psi(p_1p_2q)\le 0, \psi((e-p_1)p_2q)\le 0, \psi((e-p_2)p_1q)\le 0$, making
 $$\psi((p_1\vee p_2)q)=\psi(p_1p_2)q+\psi((e-p_1)p_2q)+\psi((e-p_2)p_1q)\le 0.$$
 Hence $p_1\vee p_2 \in{\cal H}$ and ${\cal H}$ is
 closed with respect to pairwise suprema. Thus the maximal element of ${\cal H}$ is unique.
 Denote this maximal element by $q$.
As there exists $p\in{\cal B}$ with $\psi(p)\not\ge 0$, Corollary 
\ref{existence} gives that ${\cal H}\ne \{0\}$.
Since $\psi$ is order bounded and $G$ is Dedekind complete, we can define 
 $$\beta:=\inf\{ \psi(p)\ |\ p\in{\cal H} \},$$
 and, again from Corollary \ref{existence}, $\beta<0$.
 
   We now show that $\psi(q)=\beta$.
     If this is not the case, then  $\psi(q)>\beta$ and there exists $p\in {\cal H}$ with
 $\psi(p)\not\ge \psi(q)$. Let $u:=P_{(\psi(p)-\psi(q))^-}e>0$, 
 then $u(\psi(p)-\psi(q))=-(\psi(p)-\psi(q))^-<0$ so  $u\psi(p)<u\psi(q)$.
 As $q$ is the maximal element of ${\cal H}$, $p\le q$ and $up+(e-u)q\le q$, further $u\in G$ gives that  $u$ commutes with $\psi$. Thus  
 \begin{eqnarray}
  \psi(up+(e-u)q)=u\psi(p)+(e-u)\psi(q)<u\psi(q)+(e-u)\psi(q)=\psi(q).\label{1of2}
 \end{eqnarray}
 But $u(q-p)\in{\cal B}$ with $u(q-p)\le q$ and $q$ is strongly negative with respect to $\psi$, so
 $\psi(u(q-p))\le 0$. Hence, by (\ref{1of2}),
 \begin{eqnarray*}
  \psi(q)=\psi(u(q-p))+\psi((up+(e-u)q)\le\psi((up+(e-u)q)<\psi(q), 
 \end{eqnarray*}
 a contradiction.
 Thus $\psi(q)=\beta$.

 It remains only to show that $e-q$ is strongly positive with respect to $\psi$. Suppose that this is not the case.
 Then there exists $p\in {\cal B}$ with $p\le e-q$ and $\psi(p)\not\ge 0$.
 Corollary~\ref{existence} now gives that there exists $m\in{\cal B}$ with $m\le p$ strongly negative with respect to $\psi$.
 Thus $mq=0$ and $q\vee m=q+m>q$ is strongly negative, contradicting the maximality of $q$.
 Hence $e-q$ is positive.
\qed

The Hahn-Jordan decomposition of  \cite[Theorem 3.5]{watson} follows from Theorem \ref{Hahn}
by taking:  
$E$ a Dedekind complete Riesz space with weak order unit;
$T$ a Riesz space conditional expectation operator on $E$ with $Te=e$, where $e$ is some chosen weak order unit; 
 $E$ to be $T$-universally complete;
$G=R(T)$;
$F$ an order closed Riesz subspace of $E$ with $R(T)\subset F$;
${\cal B}$ the set of all components of $e$;
and $\psi_f(p):=T(pf)$ for $p\in {\cal B}$ and $f\in E$.

%%%%%%%%%%%%%%%%%%%%%%%%%%%%%%
\newsection{A Riesz-Frechet representation theorem}\label{RF-sec}

The Hahn-Jordan decomposition of  Theorem \ref{Hahn} is used in Lemma \ref{hj-dual-lem} which provides a cruxial step for the development of the Riesz-Frechet representation theorem, Theorem~\ref{R-F}.

Let ${\mathcal{B}}$ denote the lattice of band projections on $L^2(T)$.
For brevity of notation, if $f\in E^+$ then $P_f$ will denote the band projection onto the band generated by $f$ in $E$ and $p_f:=P_fe$ where $e$ is the chosen weak order unit of $E$. Here $p_f$ is a component of $e$ in $E$. Further, due to $E$ being an $E_e$-module,  and the definition of the multiplicative structure,  $P_fg=p_fg$ where on the left is the action of the band projection $P_f$ on $g$ and on the right is the product of $p_f$ and $g$, for $g\in E$.

\begin{prop}\label{dual-map}
For $y\in E:=L^2(T)$ let 
$T_y(x):=T(xy)$ for all $x\in E$, then $T_y\in \hat{E}$ and 
\begin{equation}\label{T-operator-norm}
\|T_y\| =\|y\|_{T,2}.
\end{equation}
 The map $\Psi:y\mapsto T_y$ is $R(T)$-homogeneous, additive and injective.
\end{prop}

\begin{proof}
$T$ is linear and from the averaging property of $T$ in ${{L}}^2(T)$, see \cite{KLW-exp}, we have that $T_y$ is $R(T)$-homogeneous.  The order continuity of $T$ and of multiplication in $L^2(T)$ give that $T_y$ is in
$E^*$.  That $T_y$ is in $\hat{E}$ and that $\|T_y\| \le\|y\|_{T,2}$ follow directly from (\ref{holder}). However,
\begin{equation}\label{bound-2022}
\|T_y\|\|y\|_{T,2}\ge |T_y(y)|=\|y\|_{T,2}\|y\|_{T,2}.
\end{equation}
Let $\beta$ be the canonical partial inverse of  $\|y\|_{T,2}$ in $R(T)$, see the Appendix, we can multiply 
(\ref{bound-2022}) by $\beta$ to obtain
\begin{eqnarray*}
\|T_y\|&\ge& P_{\|y\|_{T,2}}\|T_y\|\\
&=&\beta\|y\|_{T,2}\|T_y\|\\
&\ge& \beta |T_y(y)|\\
&=&\beta \|y\|_{T,2}^2\\
&=&P_{\|y\|_{T,2}}\|y\|_{T,2}\\
&=&\|y\|_{T,2}.
\end{eqnarray*}
Thus (\ref{T-operator-norm}) holds.

That the map $y\mapsto T_y$ is linear is straight forward, $R(T)$-homogeneity follows from $T$ being an averaging operator, while injectivity follows from the strict positivity of $T$ and (\ref{T-operator-norm}).
\qed
\end{proof}

As a direct application of Proposition \ref{dual-map},  $T=T_e\in E^*$.

The following Lemma, critical to the proof of the Riesz-Frechet representation theorem, relies on 
 the Abstract Hahn-Jordan Decomposition, Theorem \ref{Hahn}.

\begin{lem}\label{hj-dual-lem}
For each ${\frak g}\in \hat{E}$, there is a component $q_{\frak g}^+$ of $e$ in $E$ so that  ${\frak g}(pq_g^+)\ge 0$ and ${\frak g}((e-q_g^+)p)\le 0$ for all components $p$ of $e$ in $E$.
\end{lem}

\begin{proof}
We take $G=R(T)$,
$\psi(p)={\frak g}(p)$ and ${\cal B}$ to be the set of all components of $e$ in $E$. Now, since $|\psi(p)|\le \|{\frak g}\|$ for all $p\in{\cal B}$, Theorem \ref{Hahn} is applicable and gives that there is $q\in{\cal B}$ which is strongly positive with respect to $\psi$ and has $e-q$ strongly negative with respect to $\psi$. Take $q_{\frak g}^+:=q$. 
\qed
\end{proof}

\begin{thm}[Riesz-Frechet representation theorem in Riesz space]\label{R-F}
For each ${\frak f}\in \hat{E}$ there exists $y({\frak f})\in E:=L^2(T)$ such that 
${\frak f}=T_{y({\frak f})}$. 
\end{thm}

\begin{proof} 
For each ${\frak g}\in \hat{E}$ take $q_{\frak g}^+$ as in Lemma \ref{hj-dual-lem}.

Let ${\frak f}\in\hat{E}$ and note
that $T\in \hat{E}$ making ${\frak f}-\frac{k}{2^n}T\in \hat{E}$ for $n\in\N$ and $k\in\N_0:=\{0\}\cup\N$.
From the definition of $q_g^+$,
$$\left({\frak f}-\frac{k}{2^n}T\right)(p)\ge 0$$
for all $p\in \mathcal{B}$ with $p\le q^+_{{\frak f}-2^{-n}kT}$.
In particular
$${\frak f}(p)\ge 2^{-n}kT(p),$$
for all $p\in \mathcal{B}$ with $p\le q^+_{{\frak f}-2^{-n}kT}$. 

Again
$$\left({\frak f}-\frac{k+1}{2^n}T\right)(p)\le 0$$
for all $p\in \mathcal{B}$ with $p\le e-q^+_{{\frak f}-2^{-n}(k+1)T}$.

Let 
$$h_k^n=q^+_{{\frak f}-2^{-n}kT}(e-q^+_{{\frak f}-2^{-n}(k+1)T})$$
then $h_k^nh_j^n=0$ for all $k\ne j$ so 
$$\sum_{k=0}^\infty  h_k^n =:q^+$$
is a component of $e$.
These limits exist in $E_u$,  due to the disjointness of the summands, and as they are bounded by $e$ they are also in $E=L^2(T)$.
Further, for all $p\in{\cal B}$,
\begin{eqnarray}\label{july-2022}
\frac{k}{2^n}T(ph_k^n)\le {\frak f}(ph_k^n)\le  \frac{k+1}{2^n}Tph_k^n,
\end{eqnarray}
i.e.,
$$0\le {\frak f}(ph_k^n)-\frac{k}{2^n}T(ph_k^n)\le  \frac{1}{2^n}Tph_k^n.$$

Note that ${\frak f}(q^+p)\ge  0$ and ${\frak f}((e-q^+)p)\le 0$ for all $p\in {\cal B}$.

Now ${\frak f}(q^+p)\le \|{\frak f}\|\in R(T)$ for all $p\in {\cal B}$ and as $R(T)$ is an $f$-algebra contained in $E=L^2(T)$ we have $({\frak f}(q^+p))^2\le \|{\frak f}\|^2\in R(T)$.
The sums
$$\sum_{k=0}^\infty \frac{k}{2^n} h_k^n=:s_n$$
and
$$\sum_{k=0}^\infty \frac{k^2}{2^{2n}} h_k^n=s_n^2,$$
exist in $E_u$ due to disjointness of the summands. Here we have also used that $(h_k^n)^2=h_k^n$.
We note that $s_n$ exists in $L^1(T)$ since 
$$T\left(\sum_{k=0}^N \frac{k}{2^n} h_k^n\right)
=\sum_{k=0}^N \frac{k}{2^n} Th_k^n
\le \sum_{k=0}^N  {\frak f}(h_k^n)
\le {\frak f}(q^+)$$
for all $N\in \N$, and hence $Ts_n\le {\frak f}(q^+)$ for all $n\in \N$.
The sequence $(s_n)$ is increasing in $L^1(T)$ and as noted above $Ts_n\le {\frak f}(q^+)$
for all $n\in \N$ so, by the $T$-universal completeness of $L^1(T)$, $(s_n)$ converges in order to some $s$ in $L^1(T)$.

Working in $E_u^+$, from (\ref{norm-bound}) and (\ref{july-2022}), we have that
\begin{eqnarray*}
 \left(\sum_{k=0}^N\frac{k^2}{2^{2n}} Th_k^n\right)^2
 &\le&\left(\sum_{k=0}^N \frac{k}{2^{n}} {\frak f}(h_k^n)\right)^2\\
 &=&\left({\frak f}\left(\sum_{k=0}^N \frac{k}{2^{n}} h_k^n\right)\right)^2\\
 &\le&\|{\frak f}\|^2  T\left(\left(\sum_{k=0}^N \frac{k}{2^{n}} h_k^n\right)^2\right)\\
 &=&\|{\frak f}\|^2  T\left(\sum_{k=0}^N \frac{k^2}{2^{2n}} h_k^n\right)\\
 &=&\|{\frak f}\|^2  \left(\sum_{k=0}^N \frac{k^2}{2^{2n}} Th_k^n\right)
 \end{eqnarray*}
 Multiplying the above by the partial inverse of $\displaystyle{\sum_{k=0}^N \frac{k^2}{2^{2n}} Th_k^n}$, we obtain
 \begin{eqnarray*}
 T\left(\sum_{k=0}^N \frac{k^2}{2^{2n}} h_k^n\right)&\le& \|{\frak f}\|^2\in R(T)^+  
 \end{eqnarray*}
 for all $N\in\N$. Thus, by the $T$-universal completeness of $L^1(T)$,
 \begin{eqnarray*}
  s_n^2=\sum_{k=0}^\infty \frac{k^2}{2^{2n}} h_k^n\in L^1(T)   
 \end{eqnarray*}
 which gives that $s_n\in L^2(T)$.
 From the above $T(s_n^2)\le \|{\frak f}\|^2$ for all $n\in\N$ and $(s_n^2)$ increases in order to $s^2$ Thus, from the 
 $T$-universal completeness of $L^1(T)$, $s^2\in L^1(T)$, making $s\in L^2(T)$.
Working from (\ref{july-2022}),
\begin{eqnarray}\label{august-2022}
T(ps_n)=\sum_{k=0}^\infty\frac{k}{2^n}T(ph_k^n)\le {\frak f}(pq^+)\le  \sum_{k=0}^\infty\frac{k+1}{2^n}Tph_k^n\le T(ps_n)+\frac{1}{2^n} T(pq_{{\frak f}}^+),
\end{eqnarray}
Taking the order limit as $n\to\infty$ in (\ref{august-2022}) gives
\begin{eqnarray}\label{august-2022-1}
T(ps)={\frak f}(pq^+).
\end{eqnarray}
Applying Freudenthal's theorem along with the order continuity and linearity of $T$ and ${\frak f}$ to (\ref{august-2022-1}) we have that
\begin{eqnarray}\label{august-2022-2}
T(gs)={\frak f}(gq^+),
\end{eqnarray}
for all $g\in E^+$. 
This extends by linearity to all $g\in E$.
We note here that $s$ is in the band generated by $q^+$. 

Applying the above to $-{\frak f}$ we have that there is $\sigma\in E^+$ 
and $q^-\in{\cal B}$ so that:
$-{\frak f}(q^-p)\ge  0$ and $-{\frak f}((e-q^-)p)\le 0$ for all $p\in {\cal B}$;
$\sigma$ is in the band generated by $q^-$; and 
\begin{eqnarray}\label{august-2022-3}
T(g\sigma)=-{\frak f}(gq^-)),
\end{eqnarray}
for all $g\in E$. 

Let $y({\frak f}):=s-\sigma$, then $y({\frak f})\in L^2(T)=E$ and 
$$T_{y({\frak f})}(g)=T(gy({\frak f}))=T(gs)-T(g\sigma)={\frak f}(g(q^++q^-)),$$
for all $g\in E$.
Now
$${\frak f}(g)={\frak f}(g(q^++q^-))-{\frak f}(gq^+q^-)+{\frak f}(g(e-q^+)(e-q^-)).$$
Here
$${\frak f}(pq^+q^-)=0={\frak f}(p(e-q^+)(e-q^-))$$
for all $p\in{\cal B}$, and thus by Freudenthal's theorem 
$${\frak f}(gq^+q^-)=0={\frak f}(g(e-q^+)(e-q^-))$$
for all $g\in E$.
Hence
${\frak f}(g)=T_{y({\frak f})}(g)$ for all  $g\in E$.
\qed
\end{proof}

Combining Proposition \ref{dual-map} and Theorem \ref{R-F} gives Theorem \ref{thm-final}.

   %%%%%%%%%%%%%%%%%%%%%%%%%%%%%
 \section{Appendix - Partial Inverses}
 If $F$ is a universally complete Riesz space with weak order unit, say $e$, then $F$ is an $f$-algebra and $e$ can be taken as the algebraic unit, see \cite[Theorem 3.6]{V-E}.
 
\begin{defs}
Let $F$ be a  universally complete Riesz space with weak order unit, say $e$, and take $e$ as the algebraic unit of  the associated $f$-algebra structure. 
We say that $g\in F$ has a partial inverse if there exists $h\in F$ such that $gh=hg=P_{|g|}e$ where $P_{|g|}$ denotes the band projection onto the band generated by $|g|$. 
We refer to $h$ as the canonical partial inverse of $g$ if, in addition, to being a partial inverse to $g$, we have that $(I-P_{|g|})h=0$, i.e.,  $h\in {\cal B}_{|g|}$,  where ${\cal B}_{|g|}$ is the band generated by $|g|$.
\end{defs}

The following result gives existence, uniqueness and positivity results concerning partial inverses and canonical partial inverses. 
We denote by ${\cal B}_f$ the band generated by $f$ and by $P_f$ the band projection onto ${\cal B}_f$.

 \begin{thm}
Let $F$ be a  universally complete Riesz space with weak order unit, say $e$, which also take as the algebraic unit of  the associated $f$-algebra structure.  Each $g\in F$ has a partial inverse $h\in F$. 
The canonical partial inverse of $g$ is unique and in this case $g$ is also the canonical partial inverse of $h$. If $g\in F^+$ then so is its canonical partial inverse.
  \end{thm}
  
  \begin{proof}
  We begin by showing the existence of a partial inverse to each $f\in F^+$.
  The approach is similar to the spectral theorem proof of \cite[Theorem 40.3]{L-Z}.
  Set
  $$\bar{f}_n=\sum_{j=0}^\infty \frac{j+1}{2^n} q_{n,j}\in {\cal B}_f$$
  and
  $$\underline{f}_n=\sum_{j=0}^\infty \frac{j}{2^n} q_{n,j}\in {\cal B}_f$$
  where $q_{n,j}=Q_{n,j}e$.
  Here
  $$Q_{n,j}=P_{\left(\frac{j+1}{2^n}e-f\right)^+}\left(I-P_{\left(\frac{j}{2^n}e-f\right)^+}\right),\quad j\ge 1,$$
  and 
  $$Q_{n,0}=P_fP_{\left(\frac{1}{2^n}e-f\right)^+}.$$
 Note that $q_{n,0}e\downarrow 0$ in order, as $n\to\infty$. 
 The convergence in order of the sums defining $\bar{f}_n$ and $\underline{f}_n$ is ensured as $F$ is universally complete
  and $q_{n,i}\wedge q_{n,j}=q_{n,i}q_{n,j}=0$ for all $i\ne j$, see \cite[Definition 47.3]{L-Z}.
  Here  $0\le \bar{f}_n-\underline{f}_n\le 2^{-n}e$ and
  $0\le \underline{f}_n\le f\le \bar{f}_n$ for each $n\in\N$. Consequently both $\bar{f}_n$ and $\underline{f}_n$ converge in order to $f$, $\bar{f}_n$ as a decreasing sequence and $\underline{f}_n$ as an increasing sequence.
  Setting $p_f:=P_fe$ we have
  $$\sum_{j=0}^\infty q_{n,j}=p_f$$
  for each $n\in\N$. 
  
  Let
  $$\bar{h}_n=\gamma_n+\sum_{j=1}^\infty \frac{2^n}{j} q_{n,j}\in {\cal B}_f$$
  and
  $$\underline{h}_n=\sum_{j=0}^\infty \frac{2^n}{j+1} q_{n,j}\in {\cal B}_f$$
  where
  $$\gamma_n=\sum_{j=n}^\infty 2^j(q_{j,0}-q_{j+1,0})\in {\cal B}_f.$$
  Here $\bar{h}_n$ is decreasing in $n$ and $\underline{h}_n$ is increasing in $n$ with
  $\underline{h}_n\le \bar{h}_n$. As such we are ensured that the sequences $\underline{h}_n$ and $\bar{h}_n$
  have order limits $\underline{h}$ and $\bar{h}$ in $F$ with $\underline{h}\le\bar{h}$.
  Further we observe that 
  $$\bar{h}_n-\underline{h}_n=\gamma_n-2^n q_{n,0}+\sum_{j=1}^\infty \frac{2^n}{j(j+1)} q_{n,j}\in {\cal B}_f.$$
  Take the integer part $$k(n)=\left[2^{2n/3}\right] +1$$ 
  then $$\frac{2^n}{k(n)(k(n)+1)}\to 0,\quad \mbox{and}\quad \frac{k(n)}{2^n}\to 0$$
  as $n\to \infty$.
  Thus 
$$\sum_{j=k(n)}^\infty \frac{2^n}{j(j+1)} q_{n,j}\le \frac{2^n}{k(n)(k(n)+1)}e\to 0$$
as $n\to\infty$.
Also
$$\gamma_n-2^nq_{n,0}+\sum_{j=1}^{k(n)-1} \frac{2^n}{j(j+1)} q_{n,j}\le P_{\left(\frac{k(n)}{2^n}e-f\right)^+}h_0\to (I-P_f)h_0=0$$
in order as $n\to\infty$.
 Combining the above gives that
 $\bar{h}_n-\underline{h}_n\to 0$ in order as $n\to\infty$. Hence $\bar{h}_n$ and $\underline{h}_n$ both converge to $h$ and 
  $$p_f-\frac{1}{2^n}h\ge p_f-\frac{1}{2^n}\underline{h}_n=\underline{h}_n(\bar{f}_n-\frac{1}{2^n}e)\ge\underline{h}_nf\le \underline{h}_n\bar{f}_n=p_f.$$
 Taking the order limit as $n\to \infty$ in the above gives that
$$hf=fh=p_f$$
and as $h\in {\cal B}_f$ , $h$ is the canonical partial inverse of $f$. 
  
   For general $g\in F$ we apply the above to $g^+$ and $g^-$ to get their canonical partial inverses which we denote $h^+$ and $h^-$ and set $h=h^+-h^-$. As $g^+\wedge g^-=0$ it follows that $h^+\wedge h^-=0$. Thus
   \begin{eqnarray*}
   gh&=&(g^+-g^-)(h^+-h^-)\\ &=&g^+h^+ - g^+h^- - g^-h^+ + g^-h^-\\
    &=&g^+h^+  + g^-h^-\\ &=&P_{g^+}e+P_{g^-}e=P_{|g|}e.
  \end{eqnarray*}

   For uniqueness of the canonical partial inverse we note that if $q,p$ are both canonical partial inverses of $f$ then $p,q\in {\cal B}_{|f|}$, so
   $$p=P_{|f|}p=pP_{|f|}e=pfq=(P_{|f|}e)q=P_{|f|}q=q$$
   showing the uniqueness.
  \qed
  \end{proof}
%%%%%%%%%%%%%%%%%%%%%%%%%%%%%%

%%%%%%%%%%%%%%%%%%%%%%%%%%%%%%%%%%%%%%%%%%%%%%%%%%%%%% bibliography %%%%%%%%%%%%%%%%%%%%%%%%%%%%%%%%%%%%%%%%


\begin{thebibliography}{22}
 \bibitem{Abr-Ali}{{\sc Y.A. Abramovich, C.D. Aliprantis},
         {\em An Invitation to Operator Theory},
         Graduate Studies in Mathematics, Volume 50, American
         Mathematical Society, 2002.}
         
         
\bibitem{AB}{{\sc C.D. Aliprantis, O. Burkinshaw}, 
{\em Positive operators}, {Academic press, 1985.}}

         
 \bibitem{ando}{{\sc T. And\^o},
         {Contractive projections in ${\cal L}_p$ spaces},
          {\em Pacific J. Math.,} {\bf 17} (1966), 391-405.}


\bibitem{AT}{{\sc Y. Azouzi, M. Trabelsi}, {$L^p$-spaces with respect to conditional expectation on Riesz spaces}, {\em J. Math. Anal. Appl.}, {\bf 447} {(2017), 798-816}}


	
 \bibitem{douglas}{{\sc R.G. Douglas},
         {Contractive projections on an ${\cal L}_1$ space},
          {\em Pacific J. Math.,} {\bf 15} (1965), 443-462.}


\bibitem{grobler-1}{
{\sc J.J. Grobler}, 
{Jensen's and martingale inequalities in Riesz spaces.}
{\em Indagationes Mathematicae},
{\textbf{25}}, 
{(2014), 275--295}
.}


	\bibitem{grobler}{
{\sc J.J. Grobler}, 
{Stopped processes and Doob's optional sampling theorem},
{\textit{J. Math. Anal. Appl.}},
{\textbf{497}}, 
{(2017), 124875}.}

\bibitem{h-dp}{{\sc C.B. Huijsmans, B. de Pagter},
         {An alternative proof of a Radon-Nikod\'{y}m theorem for lattice homomorphisms},
          {\em Acta Appl. Math.,} {\bf 27} (1992), 67-71.}
         
         
\bibitem{KLW-exp}{{\sc W.-C. Kuo, C. C. A. Labuschagne, B. A. Watson}, 
{Conditional expectations on Riesz spaces}, 
{\em J. Math. Anal. Appl.}, {\bf 303} {(2005), 509-521.}}

\bibitem{KRW}{{\sc W. Kuo, M. Rogans, B.A. Watson}, 
{Mixing processes in Riesz spaces}, 
{\em J. Math. Anal. Appl.}, {\bf 456} {(2017), 992-1004.}}

 \bibitem{kusraev}{{\sc A.G. Kusraev},
         {A Radon-Nikod\'{y}m type theorem for orthosymmetric bilinear operator},
          {\em Positivity,} {\bf 14} (2010), 225-238.}


\bibitem{LW}{{\sc C. C. A. Labuschagne, B. A. Watson}, 
{Discrete stochastic integration in Riesz spaces}, 
{\em Positivity}, {\bf 14} {(2010), 859-875.}}	


    \bibitem{L-Z}{{\sc W.A.J. Luxemburg, A.C. Zaanen},
         {\em Riesz Spaces I,}
         North Holland, 1971.}
         
          \bibitem{vardy-watson}{{\sc J.J. Vardy, B.A. Watson},
         {Markov process in Riesz spaces},
          {\em Positivity}, {\bf 16} (2012), 373-391, 393.}

\bibitem{V-E}{{\sc L. Venter, P. van Eldik},
  {Universally complete Riesz spaces and $f$-algebras},
  {\em Suid-Afrikaanse Tydskrif vir Wetenskap},
  {\bf 84} (1988), 343-346.}

\bibitem{watson}{{\sc B. A. Watson},	
	{An And\^o-Douglas type theorem in Riesz spaces with a conditional expectation,}
	{\em Positivity,} {\bf 13} (2009), 543 - 558.}      
 \bibitem{zaanen-2}{{\sc A.C. Zaanen},
         {\em Riesz Spaces II,}
         North Holland, 1983.}
 \bibitem{zaanen}{{\sc A.C. Zaanen},
         {\em Introduction to Operator Theory in Riesz Space,}
         Springer Verlag, 1997.}

\end{thebibliography}
\end{document}